\documentclass[11pt]{amsart}
\usepackage{amsmath,mathrsfs,cite}
\usepackage[ps2pdf,colorlinks=true,urlcolor=blue,
citecolor=red,linkcolor=blue,linktocpage,pdfpagelabels,bookmarksnumbered,bookmarksopen]{hyperref}
\usepackage[english]{babel}

\usepackage[left=2.4cm,right=2.4cm,top=2.85cm,bottom=2.85cm]{geometry}

\numberwithin{equation}{section}
\newtheorem{theorem}{Theorem}[section]
\newtheorem{proposition}[theorem]{Proposition}
\newtheorem{lemma}[theorem]{Lemma}
\newtheorem{remark}[theorem]{Remark}

\theoremstyle{definition}

\newcommand{\R}{{\mathbb R}}

\newcommand{\dvg}{{\rm div}}
\newcommand{\iu}{{\rm i}}

\newcommand{\eps}{\varepsilon}

\def\na{D}
\def\a{\alpha}

\title[Uniqueness of ground states for quasi-linear equations]{Uniqueness of ground states for \\ a class of quasi-linear elliptic equations}

\author[F.~Gladiali]{Francesca Gladiali}
\author[M.~Squassina]{Marco Squassina}

\address{Universit\`a degli Studi di Sassari
\newline\indent
Via Piandana 4, I-07100 Sassari, Italy}
\email{fgladiali@uniss.it}

\address{Universit\`a degli Studi di Verona
\newline\indent
Strada Le Grazie 15, I-37134 Verona, Italy}
\email{marco.squassina@univr.it}

\thanks{The second author supported by 2009: {\em ``Variational and Topological
Methods in the Study of Nonlinear Phenomena''}}


\begin{document}

\begin{abstract}
We prove the uniqueness of radial positive solutions to a class of quasi-linear elliptic problems
containing in particular the quasi-linear Schr\"odinger equation.
\end{abstract}

\maketitle

\section{Introduction and main results}

Let $N\geq 3$ and $1<p<(N+2)/(N-2)$. As know,
the uniqueness \cite{Kwong}, up to translations, of the ground state solutions
of the Schr\"odinger equation ${\rm i}\psi_t+\Delta\psi+|\psi|^{p-1}\psi=0$ in $\R^N$
plays an important r\v ole in the study of its dynamical features such as
orbital stability \cite{cazlio} and soliton dynamics \cite{brosjerr}.
Recently, many contributions were devoted to the physically relevant quasi-linear
Schr\"odinger equations, for $1<p<(3N+2)/(N-2)$,
\begin{equation}
   \label{eq.schr1}
 {\rm i}\psi_t+\Delta\psi+\psi\Delta|\psi|^2+|\psi|^{p-1}\psi=0  \quad\text{in $(0,\infty)\times \R^N$}.
\end{equation}
We refer the reader to papers \cite{AR,CJS,CLW,LiWaWa1,LiWaWa2,LiWaWa3,LiWa,poppenberg2,ruizsic}
 and to the references therein both for mathematical results and physical background
of these equations. Despite they admit various physical applications, only a few rigorous mathematical studies have been
carried out in the last decade. Especially in connections with the stability issues investigated in \cite{CJS,CLW},
the problem of obtaining a uniqueness result for the ground state solutions of  \eqref{eq.schr1} is of particular relevance.
To the authors's knowledge, two contributions have been addressed to this issue so far, namely \cite{watanabe,selvitella}.
In \cite{selvitella}, the uniqueness is obtained in
the restricted range $1<p<(N+2)/(N-2)$ and the result is perturbative in nature, namely it is obtained when the term $\psi\Delta|\psi|^2$
in \eqref{eq.schr1} is replaced by $\lambda\psi\Delta|\psi|^2$ for values of $\lambda>0$ sufficiently small.
In \cite{watanabe}, the authors get uniqueness of a class of ground states of \eqref{eq.schr1} in the range $1<p<(3N+2)/(N-2)$.
If $a:\R\to\R$ is the function defined by $a(s)=2s^2 +1$, then equation \eqref{eq.schr1} rewrites as
\begin{equation}
\label{schrod}
{\rm i}\psi_t+\dvg(a(\psi)D\psi)-\frac{1}{2}a'(\psi)|D\psi|^2+|\psi|^{p-1}\psi=0,  \quad \text{in $(0,\infty)\times \R^N$}.
\end{equation}
The goal of this manuscript is to get, for a suitable class of functions $a\in C^2(\R)$, the uniqueness of radial standing waves 
$\psi(x,t)=e^{\iu mt}u(x)$, $u>0$, of \eqref{schrod}, yielding to
the quasi-linear elliptic problem in $\R^N$
\begin{equation}
\label{problema0}
-\dvg(a(u)Du)+\frac{1}{2}a'(u)|Du|^2+mu=u^p,  \quad \text{in $\R^N$}.
\end{equation}
Physically speaking, the value of the coefficient $m>0$ of the linear term has to be interpreted as a frequency.
It is readily seen by Sobolev embedding that, if the function $a$ behaves like the power $|s|^k$ at infinity, then 
$1<p<((k+1)N+2)/(N-2)$ is the optimal range for existence of nontrivial solutions of~\eqref{problema0} in 
\begin{equation}
	\label{Xdef}
X:=\big\{u\in H^1(\R^N):\, a(u) |Du|^2\in L^1(\R^N)\big\}.
\end{equation}
For the existence of ground state solutions to \eqref{problema0} 
in the polynomial quadratic case, we refer to \cite{ruizsic,CJS}.
%Referring to solutions in the forthcoming results, we shall always mean solutions $u\in C^2(\R^N)\cap X$.
\vskip5pt
\noindent
The main results of the paper are the following

\begin{theorem}
	\label{main}
	Let $N\geq 3$, $a_1>0$ and $\psi\in C^2(\R)$. Assume that there exist $2<k<2p$ such that 
	\begin{equation*}
	1<p<\frac{(k+1)N+2}{N-2},
	\end{equation*}
	and $a(s)=a_1|s|^k+\psi(s)$, with
	\begin{equation}
		\label{general-ip-1}	
		\inf_{s\geq 0}\psi(s)>0,\qquad
		\inf_{s\geq 0} \big(k\psi(s)-s\psi'(s)\big)\geq 0.
	\end{equation}
	Furthermore, assume that
	\begin{equation}
		\label{k-mag-2}	
		\lim_{s\to+\infty}s^{\frac{2-k}{2}}\psi(s)=0,\qquad
		\lim_{s\to+\infty}s^{\frac{4-k}{2}}\psi'(s)=0,\qquad
		\lim_{s\to+\infty}s^{\frac{6-k}{2}}\psi''(s)=0.
%		\sup_{s\geq \rho} \big(|\psi(s)|+s |\psi'(s)|+s^2 |\psi''(s)|\big)<+\infty.
	\end{equation}
	Then there exists $m_0>0$ such that the problem
	\begin{equation}
	\begin{cases}
	\label{problema}
	-\dvg(a(u)Du)+\frac{1}{2}a'(u)|Du|^2+mu=u^p,  & \text{in $\R^N$}, \\
	\noalign{\vskip2pt}
	 \,\,\text{$u>0$,}          & \text{in $\R^N$}, \\
	\noalign{\vskip2pt}
	 \,\,\text{$u$ is radially symmetric,}          &
	\end{cases}
	\end{equation}
	admits a unique solution $u\in X\cap C^2(\R^N)$, up to translations, for all $m\geq m_0$.
\end{theorem}

\begin{theorem}
	\label{mainbis}
	Let $N\geq 3$, $a_1>0$ and $\psi\in C^2(\R)$. Assume that there exist $0<k\leq 2$ such that 
	\begin{equation*}
	1<p<\frac{(k+1)N+2}{N-2},
	\end{equation*}
	and $a(s)=a_1|s|^k+\psi(s)$, where $\psi$ satisfies \eqref{general-ip-1}. Furthermore,
	\begin{equation}
		\label{k-min-2}	
		0<\lim_{s\to+\infty} \psi(s)<+\infty,\qquad
		\lim_{s\to+\infty} s \psi'(s)=\lim_{s\to+\infty} s^2 \psi''(s)=0.
	\end{equation}
	Then there exists $m_0>0$ such that problem \eqref{problema}
	admits a unique solution $u\in X\cap C^2(\R^N)$, up to translations, for all $m\geq m_0$.
\end{theorem}

\noindent
Of course, in the particular but physically relevant case when $k=2$, $a_1=2$ and $\psi(s)\equiv 1$, one gets uniqueness of
ground states of the quasi-linear Schr\"odinger equation 
$$
-\Delta u-u\Delta u^2+mu=u^p\quad \text{in $\R^N$},\quad u>0\quad \text{in $\R^N$},\qquad 1<p<(3N+2)/(N-2),
$$
for $m$ large. Essentially,
thinking to more general situations where a nonlinearity $q(u)$ is considered in place of $u^p$, what we 
conjecture that should play a r\v ole is the fact that
$\sqrt{a(s)}/q(s)$ goes to zero for $s\to+\infty$,
namely the source $q$ is dominant upon the quasi-linear diffusion $\sqrt{a(s)}$.
In the study of high-frequency standing wave solutions $\psi(x,t)=e^{\iu mt}u(x)$ to \eqref{schrod}
the value of $m>0$ can be taken large, so that the condition
in the statement of Theorems \ref{main} and \ref{mainbis} is satisfied. 
\vskip2pt
\noindent
Under the assumptions of Theorems \ref{main}-\ref{mainbis}, by exploiting the sign and symmetry
properties of the set of ground state solutions, following as a variant 
of \cite[Theorem 1.3]{CJS}, we have the following

\begin{theorem}
	\label{main2}
	There exists $m_0>0$ such that, for all $m\geq m_0$, the ground state solutions $u\in X$ to
	\begin{equation*}
	-\dvg(a(u)Du)+\frac{1}{2}a'(u)|Du|^2+mu=|u|^{p-1}u  \quad\text{in $\R^N$}, 
	\end{equation*}
	are unique up to translations, positive, radially symmetric, decreasing 
	and exponentially decaying.
\end{theorem}

\noindent
According to the achievements established 
in the recent papers \cite{CJS,CLW}, for the Schr\"odinger equation \eqref{schrod}, the orbital stability range for the 
standing waves solutions associated with the problem
\begin{equation}
	\label{minproblll}
\min_{\|u\|_{L^2(\R^N)}=\gamma}{\mathscr E}(u),\quad \gamma>0,
\qquad {\mathscr E}(u)=\frac{1}{2}\int_{\R^N}a(u)|Du|^2-\frac{1}{p+1}\int_{\R^N}|u|^{p+1},
\end{equation}
is $1<p<k+1+{4/N},$ which is contained in the optimal range of the statements of
Theorems \ref{main}-\ref{mainbis}. We point out that the uniqueness, up to translations,
of the solutions to problem \eqref{problema} does not easily yield uniqueness of the solutions 
to the minimization problem~\eqref{minproblll}. This further conclusion, which is currently unavailable, 
would of course be rather useful for both analytical and numerical purposes.

\section{Dual semi-linear problem}

\noindent
Problem \eqref{problema} is formally associated with the functional 
$$
X\ni u\mapsto \frac{1}{2}\int_{\R^N}a(u)|Du|^2+\frac{m}{2}\int_{\R^N}|u|^2-\frac{1}{p+1}\int_{\R^N}|u|^{p+1}.
$$
Let $a\in C^2(\R)$ be such that $a(s)\geq \nu$ for any $s\in\R$, for some $\nu>0$.
Let $g:\R\to\R$ denote the (strictly increasing) solution to the Cauchy problem
%VECCHIA
%For a $a\in C^2(\R)$, let $g:\R\to\R$ denote the (strictly increasing) solution to the Cauchy problem
\begin{equation}
	\label{1.1}
g'(s)=\frac 1{\sqrt{a(g(s))}},\quad \quad g(0)=0.
\end{equation}
The solution is global due to the coercivity of $a$.%, namely $a(s)\geq \nu$ for any $s\in\R$, for some $\nu>0$.

\begin{lemma}
	\label{primolem}
The function $g$ is uniquely defined, $g\in C^3(\R)$ and invertible. Moreover we have
\begin{eqnarray}
&&|g(s)| \leq \frac 1{\sqrt \nu}|s|,\qquad \text{for every $s\in\R$,}\label{1.2}\\
%&& g(s)>sg'(s)\qquad \text{for every $s\geq 0$,} \label{1.3}       \\
&&\lim_{s\to 0}\frac {g(s)}s=c_0,\quad\,\, c_0:=\frac 1{\sqrt{a(0)}}>0\qquad \text{and} 
\qquad \lim_{t\to 0}\frac {g^{-1}(t)}t=\frac 1{c_0}. \label{1.4}
\end{eqnarray}
\end{lemma}
\begin{proof}
Since $a(s)\geq \nu$ for every $s\in\R$,
from \eqref{1.1}  we have that $|g'(s)|\leq \frac 1{\sqrt \nu}$
for all $s\in\R$. Integrating, we get \eqref{1.2}.
%$|g(s)|\leq \frac 1{\sqrt \nu}|t|$.
%From \eqref{1.1} we have that $ \sqrt{a(g)}g'=1$. Integrating we get
%$\int_0^t \sqrt{a(g(t))}g'(t)dt=t$ so that $\int_0^{g(t)} \sqrt{a(y)}dy=t$. Then
Furthermore, from \eqref{1.1}, we have
$$
\lim_{s\to 0}\frac {g(s)}s=\lim_{s\to 0}\frac 1{\sqrt{a(g(s))}}=\frac 1{\sqrt{a(0)}}.
$$
From \eqref{1.1} $g$ is invertible and \eqref{1.4} follows. 
%Finally let us consider the function $\Psi(s):= g(s)\sqrt{a(g)}-s$. We have $\Psi(0)=0$
%since $g(0)=0$ and $\Psi'(s)=\frac{a'(g) g}{2a(g)}>0$ from \eqref{apropp}. Then $\Psi(s)>0$ and \eqref{1.3} follows.
\end{proof}

\noindent
Using $g$, we can related, as in \cite{watanabe,AR}, problem \eqref{problema} 
to the following semi-linear equation
\begin{equation}\label{1.5}
-\Delta v+m\dfrac{g(v)}{\sqrt{a(g(v))}}=\dfrac{|g(v)|^{p-1}g(v)} {\sqrt{a(g(v))}}  \quad\text{in $\R^N$}, \qquad
v>0\quad \text{in $\R^N$},
\end{equation}
and its associated functional $I:H^1(\R^N)\to\R$ defined by
\begin{equation}\label{1.6}
I(v)=\frac{1}{2}\int_{\R^N}|Dv|^2-\int_{\R^N}H(v),\qquad
H(s)=\int_0^s h(t)dt,\quad h(t)=\frac{|g(t)|^{p-1}g(t)-mg(t)} {\sqrt{a(g(t))}}.
\end{equation}

\noindent
Consider now the following condition, namely there exists $a_1>0$ such that
\begin{equation}
	\label{1.7}
\lim_{s\to +\infty}\frac{a(s)} {s^k}=a_1.
\end{equation}

\begin{lemma}\label{l2.3}
Let $g$ be as defined in \eqref{1.1} and assume condition \eqref{1.7}. Then, we have
\begin{equation}\label{1.8}
\lim_{s\to +\infty}\frac{g(s)}{s^{\frac 2{k+2}}}=c_1, \,\qquad c_1:= \big(\frac{k+2} 2 \frac 1{\sqrt a_1} \big)^{\frac 2{k+2}}.
\end{equation}
Furthermore, it holds
\begin{equation}\label{1.8a}
\lim_{t\to +\infty}\frac {g^{-1}(t)}{t^{\frac {k+2}2}} =\Big(\frac 1{c_1}\Big)^{\frac {k+2}2}
\end{equation}
and
\begin{equation}\label{limnum}
\lim_{s\to +\infty} s(g')^2 =\left\{ \begin{array}{ll}
0 & \text{if $k>2$}\\
\frac 1{a_1c_1^k}& \text{if $k=2$}\\
+\infty &\text{if $0<k<2$.}
\end{array}
\right.
\end{equation}
\end{lemma}
\begin{proof}
From the monotonicity of $g$, and from problem \eqref{1.1} we have that $g(s)\to+\infty$ 
as $s\to+\infty$ and $g'(s)\to 0$ as $s\to+\infty$. Then, using \eqref{1.7}, we have
\begin{align*}
\lim_{s\to +\infty}\frac{g(s)}{s^{\frac 2{k+2}}} &=\lim_{s\to +\infty}\frac{g'(s)}{\frac 2{k+2}s^{\frac 2{k+2}-1}}
=\frac {k+2} 2\lim_{s\to +\infty}\frac {s^{\frac k{k+2}}}{\sqrt{a(g)}} \nonumber\\
&= \frac {k+2} 2\lim_{s\to +\infty}\sqrt{\frac {s^{\frac {2k}{k+2}}}{\frac {a(g)}{g^k} g^k}}
=\frac {k+2} 2\frac 1{\sqrt{a_1}} \lim_{s\to +\infty}\Big( \frac {s^{\frac 2{k+2}}}{g(s)}\Big)^{\frac k2}.
\end{align*}
In particular, the limit on the left hand side cannot be equal to $0$ or $+\infty$. In turn, we get
$$
\Big( \lim_{s\to +\infty}\frac{g(s)}{s^{\frac 2{k+2}}}\Big)^{1+\frac k2}=\frac {k+2} 2\frac 1{\sqrt{a_1}},
$$
and the claim follows. Moreover, we have
$$
\lim_{t\to +\infty} \frac {g^{-1}(t)}{t^{\frac {k+2}2}}=\lim_{s\to +\infty}\frac {s} {g(s)^{\frac {k+2}2}}
=\lim _{s\to +\infty}\Big[\frac {s^{\frac 2{k+2}}} { g(s)}\Big]^{\frac {k+2}2}=\Big(\frac 1{c_1}\Big)^{\frac {k+2}2},
$$
completing the proof of \eqref{1.8a}. Finally, observing that 
\begin{align*}
\lim_{s\to +\infty} s(g')^2 &=\lim_{s\to +\infty}\frac s{a(g)}=\lim_{s\to +\infty} \frac s{g^k\frac{a(g)}{g^k}}=\frac 1 {a_1}\lim_{s\to +\infty}
\frac s{\Big( \dfrac {g(s)} {c_1s^{\frac 2{k+2}}}\Big)^k c_1^k s^{\frac {2k}{k+2}}}
%\\&
=\frac 1 {a_1c_1^k}\lim_{s\to +\infty} s^{\frac {2-k}{k+2}}
\end{align*}
then \eqref{limnum} follows.
\end{proof}

\noindent
For the sake of completeness, we recall the following two propositions.

\begin{proposition}
Suppose $a(s)|s|^{-k}\to a_1$ as $s\to\infty$ and $1<p<\frac{(k+1)N+2}{N-2}$. Then $I\in C^1(H^1(\R^N))$.
\end{proposition}
\begin{proof}
Since $a(s)|s|^{-k}\to a_1$ as $s\to\infty$, arguing as for the proof of \eqref{1.8} yields 
$$
\lim_{s\to \infty}\frac{|g(s)|}{|s|^{\frac 2{k+2}}}=c_1.
$$ 
Taking into account \eqref{1.4}, and recalling that $\frac {2(p+1)}{k+2}<\frac{2N}{N-2}$, we have
$$
\lim_{s\to 0}\frac{|h(s)|}{|s|}=\lim_{s\to 0}\frac{||g(s)|^{p-1}g(s)-mg(s)|} {|s|\sqrt{a(g(s))}}
=\frac{m} {\sqrt{a(0)}}\lim_{s\to 0}\frac{|g(s)|} {|s|}=\frac{m}{a(0)},
$$
$$
\lim_{s\to \infty}\frac{|h(s)|}{|s|^{2^*-1}}=\lim_{s\to \infty}\frac{||g(s)|^{p-1}g(s)-mg(s)|} {|s|^{2^*-1}\sqrt{a(g(s))}}
=\lim_{s\to \infty}\frac{|g(s)|^{p}} {|s|^{2^*-1}\sqrt{a(g(s))}}=\frac{c_1^{p-k/2}}{\sqrt{a_1}}\lim_{s\to \infty}\frac{|s|^{\frac{2p}{k+2}}} {|s|^{2^*-1}|s|^{\frac{k}{k+2}}}=0.
$$
This yields the assertion by \cite[Theorem A.VI]{berlio}.
\end{proof}

\begin{proposition}
	\label{llinksoll}
Suppose $a(s)|s|^{-k}\to a_1$ as $s\to\infty$ and $1<p<\frac{(k+1)N+2}{N-2}$.
Let $v\in H^1(\R^N)$ be a nontrivial critical point of $I$, $v>0$ and let $u=g(v)$.
Then $u$ is a positive classical solution of \eqref{problema}.
\end{proposition}
\begin{proof}
If $v\in H^1(\R^N)$, $v>0$, is a critical point for $I$, then $v$ is a positive weak
solution of problem \eqref{1.5}. From regularity theory (see e.g.\ \cite[Lemma 1]{berlio}) it follows that $v\in C^2(\R^N)$. Hence, from regularity
and monotonicity of $g$ we have that $u\in C^2(\R^N)$ and   $u=g(v)>0$ in $\R^N$. Moreover $\na v=( g^{-1})'(u)\na u=\sqrt{a(u)}\na u$
and $\Delta v=( g^{-1})''(u)\left|\na u\right|^2 +( g^{-1})'(u)\Delta u
=\sqrt{a(u)}\Delta u+\frac {a'(u)}{2\sqrt{a(u)}}\left|\na u\right|^2 $. Using \eqref{1.5} then
$$
\frac 1{\sqrt{a(u)}} \Big[ -a(u)\Delta u -\frac 12 a'(u)\left|\na u\right|^2+mu-u^p\Big]=0,
$$
so that $u$ is a classical solution of \eqref{problema}.
\end{proof}

\noindent
Next we prove that uniqueness in $H^1(\R^N)$ for \eqref{1.5}
yields uniqueness in $X$ for the original problem, 
%DA QUI
where $X$ is as defined in \eqref{Xdef}.

\begin{lemma}
	\label{UniqLink}
Assume that condition \eqref{1.7} holds and that $1<p<\frac{(k+1)N+2}{N-2}$. Furthermore,
assume that \eqref{1.5} has at most one positive radial solution $v\in H^1(\R^N)\cap C^2(\R^N)$ up to translations.
Then \eqref{problema} admits at most one positive radial solution $u=g(v)\in X\cap C^2(\R^N)$ up to translations.
\end{lemma}

\begin{proof}
Let $u\in X\cap C^2(\R^N)$ with $u>0$ and set $v=g^{-1}(u)$. Then $v\in C^2(\R^N)$, $v>0$, and
$$
|\na v|^2 =|( g^{-1})'|^2|\na u|^2 =a(u)|\na u|^2 \in L^1(\R^N).
$$
By Lemmas \ref{primolem}-\ref{l2.3}, we have $|g^{-1}(s)| \leq C|s|+C|s|^{\frac {k+2}2}$ 
(hence $|g^{-1}(s)| \leq C|s|+C|s|^{(k+2)\frac {N}{N-2}}$), yielding
$$
\int_{\R^N}v^2\leq C\int_{\R^N} u^2+C\int_{\R^N}u^{(k+2)\frac N{N-2}}.
$$
Finally, we have for $\rho>0$ large, we obtain
\begin{equation*}
\int_{\R^N}u^{(k+2)\frac N{N-2}} \leq C_k\Big( \int_{\R^N} u^k \left| \na u\right|^2\Big)^{\frac N{N-2}} 
 \leq  C_k  \Big( \rho^k\int_{\{u\leq\rho\}}  \left| \na u\right|^2
+C\int_{\{u\geq \rho\}} a(u) \left| \na u\right|^2 \Big)^{\frac N{N-2}}.
\end{equation*}
Whence, we conclude that $v\in H^1(\R^N)$. Assume now that $u_1$ and $u_2$ are
two positive radial solutions to problem \eqref{problema} in $X\cap C^2(\R^N)$.
Then, setting $v_1:=g^{-1}(u_1)$ and $v_2:=g^{-1}(u_2)$, from the first part of the proof
we learn that $v_1,v_2\in H^1(\R^N)\cap C^2(\R^N)$ and $v_1,v_2>0$. Mimicking the proof of Proposition~\ref{llinksoll}
(argue in the reversed order), it is easy to see that $v_1$ and $v_2$ solve problem~\eqref{1.5}.
Of course, being $u_1=u_1(|x|)$ and $u_2=u_2(|x|)$, we have $v_1=v_1(|x|)$ and $v_2=v_2(|x|)$.
From the uniqueness (up to translations) for the positive radial $H^1(\R^N)\cap C^2(\R^N)$ solutions of \eqref{1.5},
it follows $v_1(\cdot)=v_2(\cdot+x_0)$, for some $x_0\in\R^N$. But then, in turn, 
$u_1(\cdot)=g(v_1(\cdot))=g(v_2(\cdot+x_0))=u_2(\cdot+x_0)$, concluding the proof.
\end{proof}

\section{General computations}

\noindent
Motivated by the conclusion of Lemma~\ref{UniqLink}, in order to get the desired uniqueness result,
in this section we study the uniqueness of positive radial solution of \eqref{1.5}. To this end, we let
$$
h(v):=\frac{|g(v)|^{p-1}g(v)-mg(v)} {\sqrt{a(g(v))}},
$$
so that we can rewrite \eqref{1.5} as $-\Delta v=h(v)$ in $\R^N$.
In this way, we have that $h(s)\leq 0$ on $(0,s_0]$, where
\begin{equation}
	\label{essezero}
s_0:=g^{-1}\big( m^{\frac 1{p-1}}\big),
\end{equation}
and $h(s)>0$ for all $s>s_0$. Since $a\in C^2(\R)$ it follows that $g\in C^3(\R)$ and $h\in C^2(0,+\infty)$.
Then, letting ${\mathscr K}_{h}(s)=\frac{sh'(s)}{h(s)}$, the uniqueness result of Serrin and Tang \cite[Theorem 1]{ST},
tells us that problem \eqref{1.5} admits a unique positive radial solution if
\begin{equation}\label{2.1}
{\mathscr K}'_{h}(s)\leq 0,\quad\text{for every $s\geq s_0$.}
\end{equation}
Observe that it holds
$$
{\mathscr K}'_{h}(s)=\frac{\left(h'(s)+sh''(s)\right)h(s)-s\left( h'(s)\right)^2}{h(s)^2}.
$$ 
Then we only need to evaluate the quantity 
$$
K(s):=\left(h'(s)+sh''(s)\right)h(s)-s\left( h'(s)\right)^2.
$$ 
We start by making some
calculations. Using (\ref{1.1}) and the fact that $g(s)>0$ for $s>0$, we have that
\begin{eqnarray}
&& h(s)=g g'\left(g^{p-1}-m\right);\nonumber\\
&&h'(s)=(p-1)g^{p-1}\left(g'\right)^2+\left(g^{p-1}-m\right)\left( \left(g'\right)^2+g g''\right);\nonumber\\
&&h''(s)=p(p-1) g^{p-2}\left(g'\right)^3+3(p-1)g^{p-1}g' g''+\left(g^{p-1}-m\right)\left( 3g'g''+g g'''\right).\nonumber
\end{eqnarray}
Then, we have
\begin{eqnarray}
&&K(s)=\left(h'(s)+sh''(s)\right) h(s)-s\left( h'(s)\right)^2\nonumber\\
&& =sg g' \left(g^{p-1}-m\right)\left[ p(p-1)g^{p-2}\left(g'\right)^3+3(p-1)g^{p-1}g' g''+\left(g^{p-1}-m\right)\left( 3g'g''+g g'''\right)\right]\nonumber\\
&&+g g' \left(g^{p-1}-m\right)\left[(p-1)g^{p-1}\left(g'\right)^2+\left(g^{p-1}-m\right)\left( \left(g'\right)^2+g g''\right)\right]\nonumber\\
&&-s\left[(p-1)g^{p-1}\left(g'\right)^2+\left(g^{p-1}-m\right)\left( \left(g'\right)^2+g g''\right)\right]^2\nonumber\\
&&= \left(g^{p-1}-m\right)^2\left[sg g' \left( 3g'g''+gg'''\right)+gg'\left(\left(g'\right)^2+g g''\right)
-s\left(\left(g'\right)^4+g^2\left(g''\right)^2+2g\left(g'\right)^2 g''\right)\right]\nonumber\\
&&+\left(g^{p-1}-m\right)\left[ sg g'\left( p(p-1)g^{p-2}\left(g'\right)^3+3(p-1)g^{p-1}g'g''\right)+(p-1)g g' g^{p-1}\left(g'\right)^2\right.\nonumber\\
&&\left.-2(p-1)sg^{p-1}\left(g'\right)^2\left(\left(g'\right)^2+g g''\right)\right]
-(p-1)^2sg^{2p-2}\left(g'\right)^4\nonumber\\
&&=\left(g^{p-1}-m\right)^2\left[sg^2g'g'''+sg\left(g'\right)^2 g''-s\left(g'\right)^4-sg^2\left(g''\right)^2+g\left(g'\right)^3+g^2g'g''\right]\nonumber\\
&&+\left(g^{p-1}-m\right)(p-1)g^{p-1}\left(g'\right)^2\left[(p-2) s\left(g'\right)^2+sgg''+gg'\right]
-(p-1)^2sg^{2p-2}\left(g'\right)^4.\nonumber
\end{eqnarray}
Now, using $g''=-\frac 12 a'\left(g\right)\left(g'\right)^4$ and
$g'''=-\frac 12 a''\left(g\right) \left( g'\right)^5+\left( a'\left(g\right)\right)^2 \left( g'\right)^7$,
we have
\begin{align*}
%& \left(h'(s)+sh''(s)\right) h(s)-s\left( h'(s)\right)^2 \\
K(s)&=\left(g^{p-1}-m\right)^2\left[sg^2g'\left(   -\frac 12 a''\left(g\right) \left( g'\right)^5+\left( a'\left(g\right)\right)^2 \left( g'\right)^7  \right)
+sg\left(g'\right)^2\left(-\frac 12 a'\left(g\right)\left(g'\right)^4\right)\right.\\
&\left. -s\left(g'\right)^4-sg^2\left( -\frac 12 a'\left(g\right)\left(g'\right)^4\right)^2+g\left(g'\right)^3+g^2g'\left( -\frac 12 a'\left(g\right)\left(g'\right)^4\right)\right] \\
&+\left(g^{p-1}-m\right)(p-1)g^{p-1}\left(g'\right)^2\left[(p-2) s\left(g'\right)^2+sg \left(-\frac 12 a'\left(g\right)\left(g'\right)^4\right)+gg'\right]
-(p-1)^2sg^{2p-2}\left(g'\right)^4 \\
&=\left(g^{p-1}-m\right)^2\left(g'\right)^3\left[-\frac 12 a''(g) s g^2 \left(g'\right)^3+\left(a'(g)\right)^2sg^2 \left(g'\right)^5 -\frac 12 a'(g)sg\left(g'\right)^3\right.\\
&\left.-sg'-\frac 14 \left( a'(g)\right)^2 sg^2 \left(g'\right)^5+g-\frac 12 a'(g) g^2\left(g'\right)^2\right]   \\
&+\left(g^{p-1}-m\right)(p-1)g^{p-1}\left(g'\right)^3\left[(p-2)sg' -\frac 12 a'(g)sg\left(g'\right)^3+g\right]
-(p-1)^2sg^{2p-2}\left(g'\right)^4 .
\end{align*}
As $(g^{p-1}-m)g^{p-1}=(g^{p-1}-m)^2+m(g^{p-1}-m)$ and $g^{2p-2}=(g^{p-1}-m)^2+2m(g^{p-1}-m)+m^2$,
\begin{align*}
%&\left(h'(s)+sh''(s)\right) h(s)-s\left( h'(s)\right)^2 \\
&K(s)=\left(g^{p-1}-m\right)^2\left(g'\right)^3\left[-\frac 12 a''(g) s g^2 \left(g'\right)^3+\frac 34\left(a'(g)\right)^2sg^2 \left(g'\right)^5 -\frac 12 a'(g)sg\left(g'\right)^3\right.\\
&\left.-sg'+g-\frac 12  a'(g) g^2 \left(g'\right)^2\right] \\
&+\left(g^{p-1}-m\right)^2(p-1)\left(g'\right)^3\left[(p-2)sg' -\frac 12 a'(g)sg\left(g'\right)^3+g\right] \\
&+m\left(g^{p-1}-m\right)(p-1)\left(g'\right)^3\left[(p-2)sg' -\frac 12 a'(g)sg\left(g'\right)^3+g\right]   \\
&-\left( \left(g^{p-1}-m\right)^2 +2m \left(g^{p-1}-m\right)+m^2\right)   (p-1)^2s\left(g'\right)^4 \\
&=\left(g^{p-1}-m\right)^2\left(g'\right)^3\left[-\frac 12 a''(g) s g^2 \left(g'\right)^3+\frac 34\left(a'(g)\right)^2sg^2
\left(g'\right)^5 -\frac 12 a'(g)sg\left(g'\right)^3-sg'+g\right.\\
&-\frac 12  a'(g) g^2 \left(g'\right)^2\left.+(p-1)(p-2)sg'-\frac 12 (p-1)a'(g) s g \left(g'\right)^3+(p-1)g -s (p-1)^2 g'\right] \\
&+m\left(g^{p-1}-m\right)(p-1)\left(g'\right)^3\left[(p-2)sg' -\frac 12 a'(g)sg\left(g'\right)^3+g-2 (p-1)s g'\right]
-m^2   (p-1)^2s\left(g'\right)^4 \\
&=\left(g^{p-1}-m\right)^2\left(g'\right)^3\left[-\frac 12 a''(g) s g^2 \left(g'\right)^3+\frac 34\left(a'(g)\right)^2sg^2 \left(g'\right)^5
-\frac 12 a'(g)g\left(g'\right)^2\left(ps g'+g\right)
-psg'+pg\right] \\
&+m\left(g^{p-1}-m\right)(p-1)\left(g'\right)^3\left[ -\frac 12 a'(g)sg\left(g'\right)^3+g-ps g'\right]
-m^2   (p-1)^2s\left(g'\right)^4  \\
&=\left(g'\right)^3\left[\left(g^{p-1}-m\right)^2{\mathscr H}_1(s)+m(p-1)\left(g^{p-1}-m\right){\mathscr H}_2(s)-m^2s(p-1)^2 g'\right]
\end{align*}
where we have set
\begin{align}
	\label{H1}
&{\mathscr H}_1(s):=-\frac 12 a''(g) s g^2 \left(g'\right)^3
+\frac 34\left(a'(g)\right)^2sg^2 \left(g'\right)^5 -\frac 12 a'(g)g\left(g'\right)^2\left(ps g'+g\right)-psg'+pg	 \\
\label{H2}
& {\mathscr H}_2(s):= -\frac 12 a'(g)sg\left(g'\right)^3+g-ps g'.
\end{align}
Notice that, since $g'>0$, we obtain
\begin{align}
	\label{formulaK}
K(s) & =m^2(g')^3\left[ \left( \frac {g^{p-1}}m-1 \right)^2{\mathscr H}_1(s)+(p-1)\left( \frac {g^{p-1}}m-1\right){\mathscr H}_2(s)-s(p-1)^2 g'\right] \\
&<  m^2(g')^3 \left( \frac {g^{p-1}}m-1 \right)\left[ \left( \frac {g^{p-1}}m-1 \right) {\mathscr H}_1(s)+(p-1){\mathscr H}_2(s)\right]. \notag
\end{align}
In particular, in order for $K$ to be asymptotically negative, it is sufficient to find $s_1>0$ and $s_2>0$
depending on the data of the problem, except on the value of $m>0$, such that
$$
{\mathscr H}_1(s)\leq 0,\quad\text{for every $s\geq s_1$},\qquad\,
{\mathscr H}_2(s)\leq 0,\quad\text{for every $s\geq s_2$}.
$$
Then one chooses the values of $m$ large enough that $s_0$ (cf.~\eqref{essezero}) becomes greater than $\max\{s_1,s_2\}$.

\section{Proof of Theorems~\ref{main} and~\ref{mainbis}}
Under the assumptions of Theorems~\ref{main} and~\ref{mainbis},
we now set $a(s)=a_0(s):=a_1|s|^k +\psi(s)$, where $k,a_1>0$ and $\psi(s)\in C^2(\R)$ is
bounded below away from zero. Observe that from \eqref{k-mag-2} in Theorem~\ref{main} 
and from \eqref{k-min-2} in Theorem~\ref{mainbis} it follows that 
$\psi(s)=o(s^k)$ as $s\to+\infty$ so that \eqref{1.7} is satisfied. 
In the case $0<k\leq 2$, \eqref{k-min-2} implies that 
$\lim\limits_{s\to +\infty}k\psi(s)-s\psi'(s)= k \lim\limits_{s\to +\infty}\psi(s)>0$.

\subsection{Asymptotic estimates}
\noindent
We have
$$
a_0(g_0)=a_1g^k_0+\psi(g_0),\quad
a'_0(g_0)=ka_1 g_0^{k-1}+\psi'(g_0),\quad
a''_0(g_0)=k(k-1)a_1g_0^{k-2}+\psi''(g_0),
$$
where $g_0$ is the function satisfying the Cauchy problem
$$
g'_0(s)=\frac 1{\sqrt{a_1g^k_0(s)+\psi(g_0)}},  \,\,\,\quad g_0(0)=0.
$$
In particular, $a_1g^{k}_0=\frac 1{(g'_0)^2}-\psi(g_0)$. We have the following simple result.

\begin{lemma}
	\label{segnolemma}
Assume that $\psi$ satisfies $\psi(s)=o(s^k)$ as $s\to+\infty$, $k\psi(s)-s\psi'(s)\geq 0$ for all $s\geq 0$ %and $k>0$ 
and  $k\psi(s)-s\psi'(s)\to\a$ as $s\to+\infty$ for some $\a>0$ in the case $0<k\leq 2$.
Consider the function $G_0:\R^+\to\R$ defined by setting
$$
G_0(s)=s-\frac2{k+2}\frac {g_0}{g_0'}=s-\frac2{k+2}g_0\sqrt{a_0(g_0)}.
$$
Then, for any $k>0$, $G_0$ is nondecreasing, 
$G_0(s)\geq 0$ for every $s\geq 0$ and $G_0(s)>0$ eventually for $s>0$ large.
Moreover, we have the limit
\begin{equation}\label{4.2}
\lim_{s\to +\infty} \frac {s(g_0')^2} {G_0(s)}=\left\{ \begin{array}{ll}
0 & \text{if  $k\geq 2$,}\\
\frac {2-k}{\a} &\text{if  $0<k< 2$}.
\end{array}
\right.
\end{equation}
\end{lemma}
\begin{proof}
We have $G_0(0)=0$ and, in addition, for every $s\geq 0$ it holds
\begin{align*}
G'_0(s)&=1-\frac2{k+2}-\frac 1{k+2} g_0\frac{a_1kg_0^{k-1}+\psi'(g_0)}{a_1g_0^k+\psi(g_0)} \\
&=\frac1{k+2}\frac 1{a_1g_0^k(s)+\psi(g_0)}\left[k a_1g_0^k(s)+k\psi(g_0)-ka_1g_0^k(s)-\psi'(g_0)g_0\right]
=\frac1{k+2}\frac {k\psi(g_0)-\psi'(g_0)g_0}{a_1g^k_0(s)+\psi(g_0)},
\end{align*}
which readily yields the first assertion. Then from \eqref{limnum} and recalling that 
$G_0$ is an increasing function, we obtain conclusion \eqref{4.2} for $k>2$. Moreover, 
in the cases $0<k\leq 2$, we can write
\begin{align*}
G'_0(s)&= \frac{1}{k+2}\frac {   k\psi(g_0)-\psi'(g_0)g_0       }{a_1g^k_0(s)\Big(1+\dfrac {\psi(g_0)}{a_1 g^k_0}\Big)}=
\frac{1}{k+2}\frac {  k\psi(g_0)-\psi'(g_0)g_0     }{a_1\Big(\dfrac {g_0(s)} {c_1s^{\frac 2{k+2}}}\Big)^k c_1^k s^{\frac {2k}{k+2}}\Big(1+\dfrac {\psi(g_0)}{a_1 g^k_0}\Big)}.
\end{align*}
From \eqref{1.8}, we get
$G'_0(s)\geq  c s^{-\frac {2k}{k+2}}$
for $s>0$ large and some $c>0$. Then, if $0<k< 2$,
$G_0(s)-C_0> c's^{(2-k)/(k+2)}$ for $s>0$ large and $c',C_0\in \R$ with $c'>0$. In the case $k=2$, we have
$G_0(s)-C_0>c' \log s$ for any $s>0$ large. Taking the limits as $s\to +\infty$ 
 we finally get, for any $0<k\leq 2$,
\begin{equation}\label{4.1}
\lim_{s\to +\infty} G_0(s)=
+\infty .
\end{equation}
Then, if $k=2$, \eqref{limnum} and \eqref{4.1} yield $\frac{s(g'_0)^2}{G_0(s)}\to 0$ as $s\to+\infty$.
If, instead, $0<k<2$ from \eqref{4.1} we have
\begin{align*}
\lim_{s\to +\infty}  \frac{s(g'_0)^2}{G_0(s)}  &=\lim_{s\to +\infty}\frac {(g'_0)^2+2sg'_0g''_0 }{G'_0(s)} \\
&=(k+2)\lim_{s\to +\infty}\frac{(g'_0)^2+2sg'_0\left( -\frac12  (ka_1g_0^{k-1}+\psi'(g_0))(g_0')^4\right)}{k\psi(g_0)-\psi'(g_0)g_0}(a_1g_0^{k}+\psi(g_0))\\
&=(k+2)\lim_{s\to +\infty}  \frac{1-s(ka_1g_0^{k-1}+\psi'(g_0))(g_0')^3}{k\psi(g_0)-\psi'(g_0)g_0}\\
&=(k+2) \lim_{s\to +\infty} \frac {1-k\dfrac s{g_0}\left(\frac 1{(g_0')^2}-\psi(g_0) \right)(g_0')^3-s\psi'(g_0)(g_0')^3 }{k\psi(g_0)-\psi'(g_0)g_0}  \\
&=(k+2) \lim_{s\to +\infty}\frac{1-k\dfrac {sg_0'}{g_0}+  \dfrac{s(g_0')^3}{g_0}\Big(k\psi(g_0) -g_0\psi'(g_0)\Big)  }{ k\psi(g_0)-\psi'(g_0)g_0 }.
\end{align*}
Moreover, we have
\begin{equation*}
\lim_{s\to +\infty}  \frac{sg'_0}{g_0}=\lim_{s\to +\infty}  \frac s{\sqrt {a_1}g_0^{1+\frac k2}}\\
=\lim_{s\to +\infty}\frac 1{\sqrt{a_1}} \frac s{ \Big( \dfrac {g_0}{c_1 s^{\frac 2{k+2}}}\Big)^{\frac {k+2}2} c_1^{\frac {k+2}2} s }=\frac 1{\sqrt{a_1}}\frac 1{c_1^{\frac {k+2}2}} =\frac 2{k+2}.
\end{equation*}
Taking into account that $k\psi(s)-s\psi'(s)\to \a>0$ as $s\to+\infty$, for $0<k<2$ we conclude
\begin{equation*}
\lim_{s\to +\infty}  \frac{s(g'_0)^2}{G_0(s)}  =\frac {k+2}{\a}\Big( 1-k \frac 2{k+2}\Big)=\frac{2-k}{\a}.
\end{equation*}
This ends the proof.
\end{proof}

\noindent
Let us now set, for each $s>0$ large, 
\begin{align*}
& {\mathscr H}_1(s):=-\frac 12 a''_0(g_0) s g^2_0 \left(g'_0\right)^3
+\frac 34\left(a'_0(g_0)\right)^2sg^2_0 \left(g'_0\right)^5 -\frac 12 a'_0(g_0)g_0\left(g'_0\right)^2\left(ps g'_0+g_0\right)-psg'_0+pg_0,	 \\
& {\mathscr H}_2(s):= -\frac 12 a'_0(g_0)sg_0\left(g'_0\right)^3+g_0-ps g'_0.
\end{align*}

\subsection{Sign of the term ${\mathscr H}_1$}

\noindent
Concerning the term ${\mathscr H}_1$, we have
\begin{align*}
{\mathscr H}_1(s)&=-\frac 12 \left[a_1 k(k-1) g_0^{k-2}+\psi''(g_0)\right]sg_0^2(g_0')^3+\frac 34 \left[a_1^2 k^2 g_0^{2k-2}+(\psi'(g_0))^2+2a_1kg_0^{k-1}\psi'(g_0)\right]sg_0^2 (g_0')^5  \\
&-\frac 12\left[a_1k g_0^{k-1}+\psi'(g_0) \right]g_0(g_0')^2(psg_0'+g_0)-psg_0'+pg_0   \\
&=-\frac 12 k(k-1)(a_1g_0^k)s(g_0')^3-\frac 12 sg_0^2(g_0')^3 \psi''(g_0)+\frac 34 k^2(a_1^2g_0^{2k})s(g_0')^5 +\frac 34 sg_0^2(g_0')^5(\psi'(g_0))^2 \\
&+\frac 32 k(a_1g_0^k) sg_0(g_0')^5 \psi'(g_0)-\frac 12 k(a_1g_0^k)(g_0')^2(psg_0'+g_0)-\frac 12 g_0 (g_0')^2 (psg_0'+g_0)\psi'(g_0)-psg_0'+pg_0  \\
&=-\frac 12 k(k-1)\Big(  \frac 1{(g_0')^2}-\psi(g_0)    \Big)s(g_0')^3-\frac 12 sg_0^2(g_0')^3 \psi''(g_0) \\
&+\frac 34 k^2\Big( \frac 1{(g_0')^4}+\psi^2(g_0)-2\frac {\psi(g_0)} {(g_0')^2}\Big)s(g_0')^5 +\frac 34 sg_0^2(g_0')^5(\psi'(g_0))^2 \\
&+\frac 32 k\Big(  \frac 1{(g_0')^2}-\psi(g_0) \Big) sg_0(g_0')^5 \psi'(g_0)-\frac 12 k\Big(  \frac 1{(g_0')^2}-\psi(g_0)  \Big)(g_0')^2(psg_0'+g_0)  \\
&-\frac 12 g_0 (g_0')^2 (psg_0'+g_0)\psi'(g_0)-psg_0'+pg_0  \\
&= -\frac12 k(k-1) sg_0'+\frac 34 k^2sg_0' -\frac 12k(psg_0'+g_0)-psg_0'+pg_0   \\
&+ \frac 12 k(k-1)s(g_0')^3\psi(g_0)-\frac 12 sg_0^2(g_0')^3 \psi''(g_0)+\frac 34 k^2 s(g_0')^5 \psi^2(g_0)-\frac 32 k^2s(g_0')^3\psi(g_0)+\frac 34 sg_0^2(g_0')^5 (\psi'(g_0))^2 \\
&+\frac 32ksg_0(g_0')^3\psi'(g_0)-\frac 32 k sg_0(g_0')^5 \psi(g_0)\psi'(g_0)+\frac 12 k(g_0')^2 (psg_0'+g_0)\psi(g_0)-\frac 12 g_0(g_0')^2(psg_0'+g_0)\psi'(g_0) \\
&=\frac 14\left(k^2+2k(1-p)-4p\right)sg_0'+\frac 12 (2p-k)g_0  \\
&+\frac 12 k \psi(g_0) s(g_0')^3(k-1-3k+p)+\frac 12 k g_0(g_0')^2 \psi(g_0)+\frac 34 k^2 s(g_0')^5 \psi^2(g_0)+{\mathscr R}_1(\psi',\psi''),  
\end{align*}
where we have set
\begin{align*}
{\mathscr R}_1(\psi',\psi'')&:=-\frac 12 sg_0^2(g_0')^3 \psi''(g_0)+\frac 34 sg_0^2(g_0')^5 (\psi'(g_0))^2+\frac 32 ksg_0(g_0')^3 \psi'(g_0)\\
&-\frac 32 ks g_0(g_0')^5\psi(g_0)\psi'(g_0)-\frac 12 g_0(g_0')^2(psg_0'+g_0)\psi'(g_0).
\end{align*}
Then
\begin{align*}
{\mathscr H}_1(s)&=\frac 14 (k+2)(k-2p)\left( sg_0'-\frac 2{k+2}g_0\right) \\
&+ \frac 12 k(p-2k-1) s(g_0')^3\psi(g_0) +\frac 12 kg_0(g_0')^2\psi(g_0)+\frac 34 k^2 s(g_0')^5 \psi^2(g_0)+{\mathscr R}_1(\psi',\psi'') \\
&=\frac 14 (k+2)(k-2p)\left( sg_0'-\frac 2{k+2}g_0\right) \\
&+ \frac 12 k (g_0')^2\psi(g_0) \left( (p-2k-1)sg_0'+g_0\right)+\frac 34 k^2 s(g_0')^5 \psi^2(g_0)+{\mathscr R}_1(\psi',\psi'') 
\end{align*}
and, since $G_0(s)\geq 0$ for all $s\geq 0$ by Lemma~\ref{segnolemma}, for every $s>0$ we get
\begin{align*}
{\mathscr H}_1(s)&\leq \frac 14 (k+2)(k-2p)g_0'G_0(s)+\frac 12 k(g_0')^2\psi(g_0)\left[(p-2k-1)sg_0'+\frac {k+2}2 sg_0' \right] \\
&+\frac 34 k^2s(g_0')^5\psi^2(g_0)+{\mathscr R}_1(\psi',\psi'') \\
&=\frac 14 (k+2)(k-2p)g_0'G_0(s)+\frac 14 k(g_0')^2\psi(g_0)(2p-3k)sg_0'+\frac 34 k^2s(g_0')^5\psi^2(g_0)+{\mathscr R}_1(\psi',\psi'') \\
&=\frac 14 g_0'G_0(s)\left[ (k+2)(k-2p)  + (2p-3k) k\frac{s(g_0')^2\psi(g_0)}{G_0(s)}+3k^2 \frac{s(g_0')^4\psi^2(g_0)}{G_0(s)} +4\frac{   {\mathscr R}_1(\psi',\psi'')   }{g_0'G_0(s)}     \right].
\end{align*}
We shall now distinguish two cases.
\vskip3pt
\noindent
{\sc\bf Case I:} ($k>2$). 
In this case, exploiting assumptions \eqref{k-mag-2}, taking into account that (cf.\ Lemma~\ref{l2.3}),
\begin{equation}
	\label{asbehav}
s(g_0')^2=O\big(g_0^{\frac{2-k}{2}}\big),\quad g'_0=O\big(g^{-\frac{k}{2}}_0\big),\quad
\text{as $s\to+\infty$}, 
\end{equation}
and recalling that $G_0$ is nondecreasing by virtue of Lemma~\ref{segnolemma}, we have
\begin{equation*}
	\lim_{s\to +\infty}\frac {s(g'_0(s))^2} {G_0(s)}\psi(g_0)=0,
	\qquad
	\lim_{s\to +\infty}\frac {s(g'_0(s))^4} {G_0(s)}\psi^2(g_0)=0,
	\qquad
	\lim_{s\to +\infty}\frac{   {\mathscr R}_1(\psi',\psi'')   }{g_0'G(s)}=0.
\end{equation*}	
Whence, since $p>k/2$, given $\eps\in (0,1)$ there exists $s_1>0$
depending upon $a$ and $k$ (but independent upon the value of $m$), such that
$$
k \big(2p-3k\big)\frac {s(g'_0(s))^2} {G_0(s)}\psi(g_0)+3 k^2\frac {s(g'_0(s))^4} {G_0(s)}\psi^2(g_0)
+4\frac{   {\mathscr R}_1(\psi',\psi'')   }{g_0'G(s)}\leq -\eps (k+2)(k-2p),\qquad\text{for every $s\geq s_1$.}
$$
Recalling Lemma~\ref{segnolemma} ($G_0(s)>0$ eventually for $s$ large), since $p>\frac{k}{2}$, it follows that
$$
{\mathscr H}_1(s) \leq \frac{1-\eps}{4}(k+2)(k-2p)g'_0(s)G_0(s)<0,
\qquad\text{for every $s\geq s_1$.}
$$
\vskip3pt
\noindent
{\sc\bf Case II:} ($0<k\leq 2$). In this case, notice that $p>1\geq k-1$ and by virtue of~\eqref{4.2} of Lemma~\ref{segnolemma}
$$
\lim_{s\to+\infty}k(2p-3k)\frac {s(g'_0(s))^2} {G_0(s)}\psi(g_0)
=\lim_{s\to+\infty} k(2p-3k)\psi(g_0)\cdot \lim_{s\to+\infty}\frac {s(g'_0(s))^2} {G_0(s)}=(2p-3k)(2-k),
$$
since, in the notations of Lemma~\ref{segnolemma}, and recalling
assumptions \eqref{k-min-2}, it holds $\a= k\psi_\infty$, where 
here we have set $\psi_\infty=\lim_{s\to+\infty} \psi(s)\in\R^+$. Of course, in turn,
$$
\lim_{s\to+\infty} \frac {s(g'_0(s))^4} {G_0(s)}\psi^2(g_0)=0.
$$
Finally, by arguing as for the previous case, taking into account assumptions~\eqref{k-min-2} we have again
$$
\lim_{s\to +\infty}\frac{   {\mathscr R}_1(\psi',\psi'')   }{g_0'G(s)}=0.
$$
In conclusion, there exists $\tilde s_1>0$
depending upon $a$ and $k$ (but independent upon $m$), with
$$
k \big(2p-3k\big)\frac {s(g'_0(s))^2} {G_0(s)}\psi(g_0)+3 k^2\frac {s(g'_0(s))^4} {G_0(s)}\psi^2(g_0)
+4\frac{   {\mathscr R}_1(\psi',\psi'')   }{g_0'G(s)}
\leq (2p-3k)(2-k)-4\eps k(k-p-1),
$$
for every $s\geq \tilde s_1$. In turn, we get
$$
{\mathscr H}_1(s)\leq (1-\eps)k(k-p-1)g'_0(s)G_0(s)<0,
\qquad\text{for every $s\geq \tilde s_1$.}
$$
In conclusion, in any case, ${\mathscr H}_1$ becomes negative for values of $s>0$ sufficiently large.

\subsection{Sign of the term ${\mathscr H}_2$}

\noindent
Concerning the term \eqref{H2}, setting
$$
{\mathscr R}_2(\psi'):= -\frac 12 \psi'(g_0)s g_0 (g'_0)^3,
$$
we have
\begin{align*}
{\mathscr H}_2(s) &=-\frac 12 a'_0(g_0)s g_0 (g'_0)^3+g_0-psg'_0 \\
& =-\frac 12 ka_1g^k_0 s(g'_0)^3+g_0-psg'_0 +{\mathscr R}_2(\psi') \\
&=-\frac 12 k \Big( \frac 1{(g'_0)^2}-\psi(g_0)\Big)s(g'_0)^3+g_0-psg'_0+{\mathscr R}_2(\psi') \\
&=g_0-\frac 12 (k+2p)sg'_0+\frac12 ks(g'_0)^3\psi(g_0)+{\mathscr R}_2(\psi'). 
\end{align*}
Then, in light of Lemma~\ref{segnolemma}, we get for $s>0$
\begin{align*}
{\mathscr H}_2(s)&\leq \Big(\frac{k+2}{2}-\frac 12 (k+2p)\Big)sg'_0+\frac12 ks(g'_0)^3\psi(g_0)+{\mathscr R}_2(\psi') \\
&=(1-p)sg'_0+\frac12 ks(g'_0)^3\psi(g_0)+{\mathscr R}_2(\psi') \\
&=sg'_0\Big[(1-p)+\frac{k(g'_0)^2\psi(g_0)}{2}+\frac{{\mathscr R}_2(\psi')}{sg_0'}\Big].
\end{align*}
Taking into account assumption \eqref{k-mag-2} (for $k>2$) and \eqref{k-min-2} (for $0<k\leq 2$), and observing 
again that \eqref{asbehav} holds, it is readily verified that it always holds
$$
\lim_{s\to+\infty}(g'_0)^2\psi(g_0)=0,\qquad
\lim_{s\to+\infty}\frac{{\mathscr R}_2(\psi')}{sg'_0}=0.
$$
Hence, since $p>1$,
there exists $s_2>0$, depending only upon $a$ and $k$, such that
${\mathscr H}_2(s)<0$ for $s\geq s_2$.

\begin{remark}\rm
By virtue of further manipulations of the terms ${\mathscr H}_1$ and ${\mathscr H_2}$, 
in the case $0<k<2$, it is possible to relax assumptions \eqref{general-ip-1}-\eqref{k-min-2} by requiring that
the following conditions hold:
\begin{align*}
&  k\psi(s)-\psi'(s)s\geq 0,\quad\text{for all $s\geq 0$},\qquad
\liminf_{s\to+\infty} k\psi(s)-\psi'(s)s\geq\alpha>0, \\
&  \lim_{s\to+\infty} \frac{k\psi(s)-\psi'(s)s}{s^k}=0,\qquad
\lim_{s\to+\infty} \frac{(k-1)s\psi'(s)-s^2\psi''(s)}{k\psi(s)-\psi'(s)s}<\frac{2k}{2-k}(p+1-k).
\end{align*}
\end{remark}

\subsection{Proof of Theorems~\ref{main} and~\ref{mainbis}  concluded}
By the previous steps, there exists
$$
\bar s=\max\{s_1,s_2\}>0,
$$
depending upon $a$ but independent upon $m>0$,
such that ${\mathscr H}_1(s)< 0$ and ${\mathscr H}_2(s)<0$ 
for any $s\geq \bar s$. Then, recalling that
$s_0:=g^{-1}( m^{1/(p-1)})$ (see formula~\eqref{essezero}), it is
sufficient to take $m_0$ sufficiently large that $s_{0}\geq \bar s$
for every $m\geq m_0$, so that (see formula~\eqref{formulaK})
$K(s)\leq 0$ for every $s\geq s_0$, yielding (condition \eqref{2.1} if fulfilled) the desired uniqueness
for problem \eqref{1.5}, and in turn for the original problem \eqref{problema}
through Lemma~\ref{UniqLink}.

\bigskip

\end{document}